\documentclass[a4paper, 11pt]{amsart}

\usepackage[utf8]{inputenc}
\usepackage[T1]{fontenc}
\usepackage[english]{babel}

\usepackage{amsmath,amssymb,amsthm}
\usepackage{enumerate}
\usepackage[shortlabels]{enumitem}
\usepackage{mathrsfs}
\usepackage{chngcntr}
\usepackage{nicematrix}
\usepackage{tikz}
\usepackage{mathtools}

\usepackage{graphicx}
\usepackage[colorlinks=true, allcolors=blue]{hyperref}

\newcommand{\rr}{\mathbb{R}}
\newcommand{\cc}{\mathbb{C}}
\newcommand{\nn}{\mathbb{N}}
\newcommand{\cinf}{C^{\infty}}
\newcommand{\cminf}{C^{-\infty}}
\newcommand{\croc}[1]{\langle #1 \rangle}

\newcommand{\mbf}[1]{\mathbf{#1}}
\newcommand{\mc}[1]{\mathcal{#1}}
\newcommand{\ve}{\varepsilon}
\newcommand{\Res}{\text{Res}}
\newcommand{\ran}{\text{ran}}

\newcommand{\mfg}{\mathfrak{g}}
\newcommand{\mf}[1]{\mathfrak{#1}}

\newcommand{\Op}{\text{Op}}

\newtheorem{thm}{Theorem}[section]
\newtheorem{lemma}[thm]{Lemma}
\newtheorem{prop}[thm]{Proposition}

\newtheorem{cor}[thm]{Corollary}

\title[Pollicott-Ruelle Resonances]{Pollicott-Ruelle resonances on flag manifolds}
\author{Alessandro Morescalchi}

\begin{document}

\thispagestyle{empty}
\vspace*{3cm}

\maketitle

\begin{abstract}
    We study the resonance spectrum of the multiflow induced on a flag manifold by the action, through multiplication by the exponential map, of the Cartan subalgebra of the underlying Lie group. We give a definition of joint resonance for the flow, then prove its discreteness and existence of resonant states. We conclude by explicit characterization of the spectrum in the special cases of Projective spaces and manifolds of full flags. 
\end{abstract}

\setcounter{page}{1}
\pagenumbering{arabic}

\section{Introduction and main results}

\subsection{Introduction}
When presented with geometrical and algebraic structures it is often instructive to study differential operators that arise naturally. These contain a great deal of information, that in many cases can be encoded in their spectra and be seen in the behaviour of the dynamics that they generate. In the case of non elliptic operators it happens however that the classical spectrum is mostly essential, so that the information is fundamentally inaccessible. To circumvent this difficulty it is sometimes possible to define a so called \emph{resonance spectrum}, whose elements have been shown in many instances to hold geometric information of interest. Examples include the computation of the correlation spectrum, \cite{FRS} \cite{liverani} \cite{DR}, relations to the dynamical zeta function \cite{dzaltro} \cite{GiLiPo}, and topological invariants \cite{DR}\cite{DR2}\cite{dzzeta}.

Consider a non-compact semisimple Lie group $G$, with a Cartan algebra $\mf{a}\subset \mfg$ and a flag manifold $F$. The homogeneous space structure induces a family $\phi$ of commuting flows on $F$ indexed by $\mf{a}$, we reference \cite{DKV} for a good introduction. 
This flow exists alongside the natural action induced by multiplication of the maximal compact subgroup $K$, and in fact the two together with the action of the nilpotent group determine the action of $G$ through the Iwasawa decomposition, and this has been studied in detail in \cite{PS}. As expected from the very different natures of $K$ and $\exp \mf{a}\subset G$, the multi-flow $\phi$ possesses a different structure, that at first sight appears as a symmetry-breaking of the homogeneous space. Under generic hypotheses on $A\in\mf{a}$, the flow $\phi^A$ has a finite number of critical points with mutually transversal stable and unstable manifolds, that is the flow is Morse-Smale. These flows are gradient flows of a set of functions fundamental for the study of oscillatory integrals appearing in the harmonic analysis of the group, moreover they reveal the structure of the manifold in that stable and unstable manifolds correspond to Schubert cells in the Bruhat decomposition of the flag manifold.

The goal of this work is to use analytical methods to construct a theory of joint resonance spectrum for this dynamical system. These higher rank resonances take the name  of \emph{Pollicott-Ruelle-Taylor resonances}, and have been first defined in \cite{GGHW} as an extension of the classical theory of resonances to a setting where the action is naturally of higher rank.
Pollicott-Ruelle resonances for Morse-Smale flows have been extensively studied in the works of N.V.Dang and G.Rivière \cite{DR}, \cite{DR1}, \cite{DR2}, who arrived at a complete description of the resonant set in terms of the Lyapunov exponents of the flow. Adapting some results of \cite{DKV} it is possible to compute the Lyapunov Exponents, from which we deduce an explicit expression for the spectrum of resonances for any $A\in\mf{a}$. In order to avoid flooding the introduction with notations, we just say that we can compute explicitly some subset $\Res\subset \text{span}_{\mathbb{N}}\Delta^+$, where $\Delta^+\subset \mf{a}^*$ is a set of positive roots for $\mfg$, such that the set of Pollicott-Ruelle resonances for the flow $\phi^A$ generated by any $A\in\mf{a}_{++}$ is
\begin{equation}\label{schifonanze}
    \Res(A)\ = \ \left\{ \alpha(A): \ \alpha\in \Res \right \}. 
\end{equation}
See Proposition \ref{risonanze} for the exact expression of $\Res$. 
This allows us to recover information on the correlation function of the flow. Without delving into details, by taking a $f\in\cinf(F)$ and $g\in \Omega^d(F)$ we can define
\[
C_{f, g}^A(t) = \int_F  \phi^A_{t*}(f)\ g . 
\]
The correlation function has many properties, for example the Laplace transform of $C_{f, g}$ is a meromorphic function with poles centered at the Pollicott-Ruelle resonances of the flow. Following \cite{DR}, there exist spectral projectors $\pi_{ \lambda}: \cinf(F)\rightarrow \cminf(F)$, for $\lambda\in \Res(X_A)$ and independent on $f, g$, such that 
\[
    C_{f,g}(t)= \sum_{\lambda\in\Res(X_A) \ |\lambda|<\Lambda} e^{t\lambda} \int_F \pi_{\lambda}(f)g + O_{f, g}(e^{-\Lambda t}).
\]

We can see how the realization of resonances for each individual flow descends from the higher rank object $\Res\subset\mf{a}^*$. It is natural to ask if it is possible to establish a framework where $\Res$ emerges directly from the action $\phi: \mf{a}\rightarrow \text{Diffeo}(F)$.
The setup is very similar to the one developed in \cite{GGHW}. We take the Lie algebra morphism $\mbf{X}: \mf{a}\rightarrow \cinf(F, TF)$ sending each $A$ to the vector field generating $\phi^A$, and define the multi-spectrum in terms of cohomology of an exact sequence generated by $\mbf{X}$. This definition is the most natural generalization of the 1-dimensional case and we give a complete characterisation. Again, to avoid a notation-heavy introduction we give a softer version of the theorem, see Theorem \ref{grosso} for the full statement. 
\begin{thm}
    The Pollicott-Ruelle-Taylor spectrum is discrete, and to each element of the spectrum is associated at least one resonant state in $\cminf_{E^*_u}$:
    \begin{equation*}
        \lambda\in\Res(\mc{C}) \ \Longleftrightarrow \ \exists u\in\cminf_{E^*_u}: \ (\mbf{X}-\lambda)u=0.
    \end{equation*}
    Moreover, the set of resonances can be computed explicitly and corresponds to the resonances found in \eqref{schifonanze}   :    
    \begin{eqnarray*}
        \Res(\mc{C})= \Res. 
    \end{eqnarray*}
\end{thm}
We then restrict ourselves to the special case of "classical" flag manifolds, i.e. $G=SL(n, \rr)$. We are able to give completely explicit description of the resonance spectrum of Projective spaces and manifolds of full flags. We remark general the resonance spectrum might coincide asymptotically for most flag manifolds. Nevertheless, if one restricts to Grassmanians, it is possible to differentiate manifolds based on the spectra in some cases.

\subsection{Acknowledgements}

The author is deeply grateful to Yannick Guedes Bonthonneau for the suggestion of this topic, for his guidance and for many helpful insights. His patience and availability were invaluable. The author would also like to thank Lasse Lennart Wolf for the many interesting discussions, which despite being not directly related were very important for the author's own understanding on the subject.

\section{Dynamical Setting}\label{setting}

We give a brief introduction to the notation and the basics of the system we are dealing with. The whole setting is fundamentally algebraic in nature and as often happens in the study of Lie groups one finds itself with a variety of interacting objects indexed by some fundamental parameters of the group. At the same time, the argument is mostly based on the already established geometric properties and we will not need to dive deep in the interactions and definition of the various components, as the resulting geometry is sufficient for our study. The reader who is ready to accept this section should have no difficulty in understanding the rest of the paper, even if unfamiliar with the context of flag manifolds. 

Let $G$ be a non-compact semisimple Lie group, $\mfg=\text{Lie}(G)$. Fix $\mathfrak{a}\subset\mfg$ a Cartan Subalgebra, with root system $\Delta$, positive roots $\Delta^+$, and Weyl group $\mathfrak{w}$. The Iwasawa decomposition is $\mfg = \mathfrak{k}\oplus \mathfrak{a}\oplus \mathfrak{n}^+$ allows to associate flag manifolds of $G$ with subsets of $\mf{a}$. 
Given a set $\mathcal{C}\subset \mathfrak{a}$, we define 
\begin{eqnarray*}
\Delta^0(\mathcal{C}) = \{ \alpha\in\Delta: \ \alpha(\mathcal{C})=0\},  \quad 
\Delta^{\pm}(\mathcal{C})=\{ \alpha\in\Delta: \ \pm \alpha(\mathcal{C})>0\}.
\end{eqnarray*}
Then the parabolic algebra is
\[
\mathfrak{p}_{\mathcal{C}}= \mathfrak{k} \oplus \mathfrak{a} \oplus \mathfrak{n}^+ \ +  \sum_{\alpha\in\Delta^0(\mathcal{C})} \mathfrak{g}_{\alpha} = \mathfrak{k}\oplus\mathfrak{a}\oplus \sum_{\alpha\in\Delta^0(\mathcal{C})\cup\Delta^+(\mathcal{C})} \mathfrak{g}_{\alpha}
\]
where $\mathfrak{g}_{\alpha} = \{H\in \mathfrak{g}: \ [H, X]=\alpha(X)H \quad \forall H\in \mathfrak{a}\}$ is the eigenspace associated to $\alpha$ for the adjoint action. The parabolic group is the normaliser of $\mathfrak{p}_{\mathcal{C}}$ in $G$, $P_{\mathcal{C}}=N_G(\mathfrak{p}_{\mathcal{C}})$, so that the flag manifold is the compact homogeneous space $F_{\mathcal{C}}=G/P_{\mathcal{C}}$.

The left action of $G$ on $F_{\mathcal{C}}$ induces a commuting family of flows
\[
\forall A\in\mathfrak{a} \qquad \phi^A_t: xP\longmapsto e^{tA}xP. 
\]
We also associate the map $\mbf{X}:\mf{a}\rightarrow \cinf(F_{\mc{C}}, TF_{\mc{C}})$ given by
\[
\forall A\in\mf{a} \qquad \mbf{X}_A = \frac{d}{dt}\phi^A_t
\]
We are interested in the uniformity properties for the commuting action restricted to the positive Weyl chamber $\mathfrak{a}_{++}=\{ A\in\mathfrak{a}: \ \alpha(A)\geq0\ \forall \alpha\in\Delta^+\}$. Since the Weyl group acts transitively on the set of Weyl chambers, the results of this work will hold for any chosen Weyl chamber. 

There exists some Riemannian metric $\beta$ on $F_{\mathcal{C}}$ such that $\phi^A$ is the gradient flow induced by the function $f_A\in\cinf(F_{\mathcal{C}})$. In full generality the flow is of Morse-Bott type, but if
\begin{equation} \label{condizione}
    \forall \alpha\in\Delta: [  \alpha(A)=0 ] \Rightarrow [ \alpha(w\mathcal{C})=0 \quad \forall w\in\mathfrak{w}]
\end{equation}
then $\phi^A$ is Morse-Smale, in the sense of \cite{DR}, see Proposition 1.5, Corollary
3.7  and Lemma 4.2 in \cite{DKV}. This condition can be proven by taking a finite set of explicit coordinates where the flow is linear, see \eqref{coordinate}, and \eqref{condizione} implies that its critical set is finite. The Morse-Smale property descends in turn from transitivity of the $G$-action on the manifold. 
To be precise, if \eqref{condizione} is satisfied, then the critical set of $\phi^A$ is the finite set $\text{Crit}(\phi)=\mathfrak{w}P_{\mathcal{C}} \subset F_{\mathcal{C}}$, the stable and unstable manifolds are 
\[
W^{s, u}(wP_{\mathcal{C}})=\exp \mf{n}^{\mp}wP_{\mathcal{C}}
\]
and these are transversal. It is a remarkable property that the stable and unstable manifolds coincide for all flows in the chamber, and this will be fundamental in the study of the multiflow. 
Condition \eqref{condizione} is trivially satisfied if $A$ is regular, that is $\alpha(A)\neq0 \quad \forall \alpha\in\Delta$. This has always been the only case studied in the literature and it seems that no investigation on other possible situations has been carried out. Upon further analysis this is not the only configuration to satisfy \eqref{condizione} but it almost is, in the sense that the regular $A$ exhaust all possible dynamics. The rigorous derivation is elementary but slightly involved, and is left to Appendix \ref{algebra}, here we only explain the emerging picture. First of all, if $G$ is simple, then $A$ must be regular for \eqref{condizione} to hold. Suppose the Lie algebra $\mfg$ decomposes in mutually commuting simple Lie algebras $\mfg=\bigoplus_{i=1}^n \mfg_i$, so that $G=\prod_i G_i$, and let $\Pi_k:\mfg\rightarrow \mfg\ominus\mfg_k$ be the projection along $\mfg_k$. If $\Pi_k\mc{C}\neq 0$, then $A_k=A-\Pi_kA$ must be regular inside $\mfg_k$. On the contrary, if the $k$-th part of $\mc{C}$ is null, then $\mfg_k\subset \mf{p}_{\mc{C}}$ and the $k$-th component of $A$ will not influence the dynamics of the flow, in particular it is allowed to be singular.
This can be formulated in the following way.
\begin{lemma}
    Suppose $\Pi_k\mc{C}=0$, then $\ker\mbf{X}\supset\mfg_k$. Moreover, the flag manifold $F_{\mc{C}}$ is isomorphic to the flag manifold $F_{\Pi_k\mc{C}}$ of the Lie group $\tilde{G}= \prod_{i\neq k} G_i$, dynamics induced by $A=\sum_i A_i$ on $F_{\mc{C}}$ are conjugated to the dynamics induced by $\tilde{A}=\Pi_kA_i$ on $F_{\Pi_k\mc{C}}$.
\end{lemma}
In such a case, one could even extend the positive Weyl chamber $\mf{a}_{++}$ to a larger subset $\mf{a}_{++}(\mc{C})=\mf{a}_{++}+\mfg_k$, noting that in fact the flow is independent of the $k$-th component. This reduction can be repeated on any $k$ such that $\Pi_k\mc{C}=0$ so that we can assume without loss of generality that $A$ is regular.

We go back to general properties of the flow, and drop the homogeneous structure notations for more convenient dynamical ones. 
For all $x\in F_{\mc{C}}$ there exist limit points $x^+, x^- \in \text{Crit}(\phi)$ such that
 \begin{equation*}
       \lim_{t\rightarrow+\infty}\phi^A_t(x)=x^+ ,\quad 
       \lim_{t\rightarrow-\infty}\phi^A_t(x)=x^-, \qquad \forall A
\end{equation*}
Let us define
\begin{eqnarray*}
    E_s(x) &=&T_xW^s(x^+)=\mathfrak{n}^+x, \\
    E_u(x)&=&T_xW^u(x^-)=\mathfrak{n}^-x,
\end{eqnarray*}
and by duality let us define $E^*_s, E^*_u\subset T^*M$ as
\begin{eqnarray*}
    E^*_s(x)&=&(E_s(x))^{\perp}\subset T_x^*F_{\mc{C}}, \\
    E^*_u(x)&=&(E_u(x))^{\perp}\subset T_x^*F_{\mc{C}}.
\end{eqnarray*}
We note that by transversality $E^s+E^u=TF_{\mc{C}}$ and therefore $E^*_s\cap E^*_u=\varnothing$.
If $S^*F_{\mc{C}}$ is the co-sphere bundle, let us also define 
\[
\Sigma_{s, u}=S^*F_{\mc{C}}\cap E^*_{u, s}
\] (notice that the indices are swapped). A very important property, whose proof can be found in \cite{DR} is that $\Sigma_s$ and $\Sigma_u$ are compact. 
Since $\phi^A_t$ is linear on each fibre, it naturally induces a flow on $S^*F_{\mc{C}}$ by $\tilde{\phi}^A_t(x, \xi)=\left(\phi^A_t(x), \frac{\phi^A_t(\xi)}{|\phi^A_t(\xi)|}\right)$. 

For all $A\in\mf{a}_{++}$, the elements of $E^*_u$ grow exponentially in positive time and shrink to $0$ in negative time, while the opposite is true for $E^*_s$. In particular, $\Sigma_s$ is a sink and $\Sigma_u$ is a source for the flow on the co-sphere bundle. 
A useful result will be the existence of positively invariant neighbourhoods of $\Sigma_s$. This is not at all surprising but the proof seems to require a bit of care, and we leave it to Appendix \ref{intorno}. 
\begin{prop} \label{intornofisso}
     Let $V_0^s\supset\Sigma_s$ be an open set. Then there exists $V^s\subset V_0^s$ containing $\Sigma_s$ that is positively flow-invariant for all $A\in\mf{a}_{++}$. 
\end{prop}

A remarkable property of the system is the existence of a natural atlas of $F_{\mc{C}}$ that linearizes all the flows $\phi^A$. Without delving into unnecessary details, for each $wP\in \text{Crit}(\phi)=\mf{w}P$, there exists a diffeomorphism $\gamma_w: \rr^{w\Delta^+(\mc{C})}\rightarrow F_{\mc{C}}$ such that the flow acts diagonally:
\begin{equation}\label{coordinate}
(\gamma_w)^*\phi^A_t(x, \xi)=(e^{t\alpha(A)}x_{\alpha}, e^{-t\alpha(A)}\xi_{\alpha})_{\alpha\in w\Delta^+(\mc{C})} \qquad \forall A\in\mf{a}.
\end{equation}
Moreover, stable and unstable manifolds of $w$ in these charts are identified with
\[ W^{s, u}(wP)=\{ x \ | \  x_{\alpha}=0 \quad \forall \alpha\in w\Delta^+(\mc{C})\cap \Delta^{\pm}\}. \]

By identifying the distance of $A$ from the boundary of $\mf{a}_{++}$ by $h(A)=\min\{ \alpha(A)\ | \ \alpha\in (\mf{w}\Delta^+(\mc{C}))\cap \Delta^+\}$ we can use these charts to infer uniform growth and degrowth estimates in terms of $h(A)$.

\section{Escape function and Sobolev Spaces}

\subsection{Escape Function}
The escape function is a delicate object depending on deep dynamical properties of the flow. Our objective is to prove the following:

\begin{prop}\label{fuggitesciocchi}
    Fix $A_0\in\mf{a}_{++}$, let $\tilde{V}^u $ be conical neighbourhood of $E^*_u$. Then, there exist constants $c_0, R>0$, a conical neighbourhood $\mc{U}\ni A_0$ and functions $m\in\cinf(T^*F_{\mc{C}}, \rr)$, $G=m\log(1+\|\xi\|^2)$ such that:
    \begin{itemize} 
        \item [(1)] $m$ is positively homogeneous of degree 0 for $\|\xi\|\geq 1$; 
        \item [(2)] $m<-1$ close to $E^*_u$, $m\geq 1 $ outside of $\tilde{V}^u$; 
        \item [(3)] $X_Am\leq 0$ on $T^*F_{\mc{C}}$ for all $A\in\mc{U}$;
        \item [(4)] $X_AG\leq -c_0$ on $T^*F_{\mc{C}}$ for $\|\xi\|\geq R$, for all $A\in\mc{U}$.
    \end{itemize}
\end{prop}
Remark that the escape function is unusually strong, due to the Morse-Smale property, as in \cite{DR}. To be more precise, we will build the escape function as
\[
G(x, \xi) = c(Nm_2(x, \xi)-f)\log(1+\|\xi\|^2).
\]
The function $m_2$ will be constructed to be strictly increasing along the flow away from $\tilde{V}^s\cup\tilde{V}^s$, for $\tilde{V}^s$ some small conical neighbourhood of $E^*_s$, and to satisfy (the opposite of) properties $(1), (2)$. We add $f_A$ to gain degrowth along $\tilde{V}^u\cup\tilde{V}^s$ away from critical points, and choose $N$ large and $c$ accordingly to keep $(2)$. Finally, natural properties of the flow allow to push the degrowth up to $\tilde{V}^u\cup\tilde{V}^s$ near critical points. Remark that generally in the study of hyperbolic diffeomorphisms one can only construct $m_1$, while the addition of $f_A$ and of the latter degrowth near critical points are possible only due to the particular setting of gradient Morse-Smale flows. A final remark worth mentioning is that the degrowth properties of $m_1$ and $\log(1+\|\xi\|^2)$ can be taken to be uniform on the set $\{A\in\mf{a}_{++}:\  h(A)>\ve\}$ for any $\ve$. The restriction to $A\in\mc{U}$ in $(3), (4)$ is present because it is not clear if all elements $A\in\mf{a}_{++}$ satisfy $X_Af_{A_0}\geq 0$ in general and $X_Af_{A_0}\geq c h(A)$ away from critical points. These properties seem natural but the author was not able to prove them in general, and didn't find them in the literature. If they were to be proven, the existence of a unique escape function for all $A\in\mf{a}_{++}$ would simplify the proof.

We begin by building $m$. 
\begin{prop}
    Let $V^s, V^u\subset S^*F_{\mc{C}}$ be neighbourhoods of $\Sigma_s, \Sigma_u$, respectively. 
    There exists $\mathcal{W}^s\subset V^s$, $m_2\in\cinf(S^*F_{\mc{C}})$ such that:
    \begin{itemize}
        \item $m_2 \leq -1$ in $S^*F_{\mc{C}}\setminus V^s$;
        \item $m_2\geq 1$ in $\mathcal{W}^s$;
        \item $\overline{X_A}m_2\geq0$ for all $A\in\mf{a}_{+}(\mc{C})$;
        \item There exists $\eta>0$ such that $\overline{X_A}m_2\geq \eta h(A)$ in $S^*F_{\mc{C}}\setminus (V^u\cup V^s)$, for all $A\in\mf{a}_{+}(\mc{C})$. 
    \end{itemize}
    Here $\overline{X_A}$ refers to the vector field induced by $X_A$ on $S^*F$, and $|A|$ is any norm in $\mf{a}$.
\end{prop}
\begin{proof}
We can suppose that $V^s,  V^u$ are flow-invariant in the future and past, respectively, for all $A\in\mf{a}_{++}$. 
Fix $A_1, ..., A_d$ dual to $\Delta^+$, so that their simplicial cone is the positive Weyl chamber $\mf{a}_{++}$ and $\alpha_i(A_j)=\delta_{i, j}$. Take $T$ to be the supremum of $T(A, p)=\inf\{ t\geq 0: \tilde{\phi}^A_t(p)\in V^s\}$ over the compact set $\{A\in\mf{a}_{+}(\mc{C}): \ h(A)=1\} \times (S^*F_{\mc{C}}\setminus V^u)$, so that $\phi^A_{T/h(A)}(p)\in V^s$ for all $p\in S^*F_{\mc{C}}\setminus V^u$. 

Let us now consider a function $m_0\in\cinf(S^*F_{\mc{C}})$ such that $m_0=1$ outside $V^s$ and $m_0=0$ near $\Sigma_s$.

Define 
\[ m_1(x, \xi)=\int_{[-T, T]^d} m_0\circ\phi^{ A_1}_{t_1}\circ ... \circ\phi^{ A_d}_{t_d} (x, \xi) dt. \]
For any $A\in\mf{a}_{++}$, let $F_A=\overline{X_A}m_1=\sum_i \alpha_i(A)F_{A_i}$. 
Compute $F_{A_d}$, and by symmetry all the $F_{A_i}$ will have the same form:
\begin{eqnarray*}
    F_{A_d}(x, \xi) &=& \frac{d}{dt}m_1\circ \phi^{A_d}_t(x, \xi) \\
    &=& \int_{[-T, T]^{d-1}} m_0(\phi^{ A_1}_{t_1}\circ .. \circ\phi^{ A_d}_{T})-m_0(\phi^{ A_1}_{t_1}\circ .. \circ\phi^{ A_d}_{-T})dt \\ 
    &=& \int_{[-T, T]^{d-1}} m_0(\phi^{A_d}_T(p_t))-m_0(\phi^{A_d}_{-T}(p_t))dt
\end{eqnarray*}
where we are writing $p_t=\phi^{ A_1}_{t_1}\circ \dots \circ\phi^{ A_{d-1}}_{t_{d_1}}(p)$. 
First remark that the integrand is non-negative pointwise. Then, if $F_A(p)=0$ the integrand has to be constantly zero and in particular $ \{F_A=0\}\subset V^s\cup  V^u$. One can see that outside of $V^u\cup V^s$ we have $X_Am_1\geq \eta h(A)$. 

If $p\notin V^s$, we have that $\phi^{A}_t(p)\notin V^s$ whenever $t\leq 0$ and $A\in\mf{a}_{++}$, in particular $m_1(p)\leq (2T)^d-T^d$. Moreover, by taking $\mc{W}^s = \phi^{A_1}_T\circ...\circ\phi^{A_d}_{T_d} ( \{m_0=1\})$, we have that $m_1=(2T)^d$ on $\mc{W}^s$.  
By taking $m_2(x, \xi) = 2(m_1(x, \xi)/T^d-2^d)+1$, we complete the proof. 
\end{proof}

The following two lemmas allow us to complete the proof. Lemma \ref{bassezza} follows from the Taylor expansion of the vector field around critical points, which permits to give bounds of the type $d(x, w)h(A)/C\leq |X_Af_{A_0}|\leq Cd(x, w)h(A)$ near the critical point $w$ and for any $A\in\mf{a}$. . The proof of lemma \ref{zioss} is a simple generalisation of the proof of the analogue result in \cite{DR}. 
\begin{lemma} \label{bassezza}
There exists a small enough conical neighbourhood $A_0\in\mc{U}\subset \mf{a}_{++}$ such that $\forall A\in\mc{U}$,  $X_A(f_{A_0})\geq 0$, and $X_Af_{A_0}\geq c(r)h(A)$ if at distance $r$ from critical points. 
\end{lemma}
\begin{lemma} \label{zioss}
    There exists $C>0$ such that, for $r>0$ small enough, on a small neighbourhood of  $E^*_s\cap T^*(\bigcup_{\mf{w}}B(w, r))$, for all $ A\in\mf{a}_{++}$ we have:
    \begin{equation} \label{Claim}
    X_A(|\xi|^2)<-Ch(A)|\xi|^2 .
\end{equation}
    Likewise, on a small neighbourhood of $E^*_u\cap T^*(\bigcup_{\mf{w}}B(w, r))$:
    \begin{equation}  \label{Claim2}
    X_A(|\xi|^2)>Ch(A)|\xi|^2.
\end{equation}
\end{lemma}

\begin{proof}{of proposition \ref{fuggitesciocchi}.}

Let $r>0$ and $B = T^*\bigcup_{w\in \text{Crit}(\phi)} B(w, r)$ be given by Lemma \ref{zioss}. Now take $\tilde{V}^s$ small enough and eventually resize $\tilde{V}^u$ so that conditions \eqref{Claim} and \eqref{Claim2} hold on $B\cap \tilde{V}^s$ and $B\cap \tilde{V}^u$, respectively. 

Now define $m_2(x, \xi)$ to be $1$ for $|\xi|\leq 1/2$ and $m_2(x, \xi) = - m_1(x, \xi/|\xi|)$ for $|\xi|>1$. Beware of reversed indices: since $m_1<-1$ outside of $V^s$ and since $V^s=\pi(\tilde{V}^u)$, then $m_2>1$ outside $\tilde{V}^u$.
Then for any two constants $c, N$ let $m(x, \xi)$
\begin{eqnarray*}
    m(x, \xi) &=&  Nm_2(x, \xi) - f_{A_0}(x) \\
    G(x, \xi) &=& m\log(1+|\xi|^2).
\end{eqnarray*}
Property $(1)$ is clearly satisfied. Near the critical sets we have
\begin{eqnarray*}
    m(x, \xi) &\leq& -N+\|f_{A_0}\|_{\mc{C}^0} \quad \text{near} \quad \tilde{V}^u \\
    m(x, \xi) &\geq& N- \|f_{A_0}\|_{\mc{C}^0} \quad \text{outside} \quad \tilde{V}^u
\end{eqnarray*}
So that if $N$ is large enough, $(2)$ is satisfied. 
For any $A\in\mc{U}$, for $\mc{U}$ given by Lemma \ref{bassezza}:
\[
X_Am(x, \xi) = -NX_Am_2(x, \xi)-X_Af_{A_0}(x)\leq 0,
\] 
hence (3) is also satisfied. Moreover, we can also infer that 
\[
X_Am(x, \xi)\leq -ch(A)
\]
for $A\in \mc{U}$ and $(x, \xi)\notin (\tilde{V}^u\cup\tilde{V}^s)\cap B$. Hence, we compute
\[
X_AG(x, \xi) = X_Am(x, \xi)\log(1+|\xi|^2) + m(x, \xi)\frac{X_A|\xi|^2}{1+|\xi|^2}
\]
and by Lemma \ref{zioss} we conclude that (4) holds. 
\end{proof}

\subsection{Sobolev Spaces}\label{sobolev}
The functional setting of the proof is based on a class of Anisotropic Sobolev spaces. The technique is standard in hyperbolic dynamics, and the papers \cite{FRS}, \cite{FS}, \cite{dzaltro} are standard references for the microlocal formulation of these spaces.

For any (asymptotically) positively homogeneous function of order 0 $m\in\cinf(T^*F_{\mc{C}}, \rr)$, denote by $S^{m(\cdot)}$ and $\Psi^{m(\cdot)}$ the classes of symbols and operators of variable order $m$. Check \cite{mahd} for the theory of variable-order pseudo-differential operators. 

Let $G(x, \xi)=m(x, \xi)\log\croc{\xi}$ be an escape function given by Proposition \ref{fuggitesciocchi}. 
Fix any quantization $\Op: S^{\infty}\rightarrow \Psi^{\infty} $ and define the pseudo-differential operator
\[
\tilde{\mc{A}}_N= Op(e^{NG})=Op(\croc{\xi}^{N m})\in\Psi^{Nm(\cdot)+}.
\]

The operator $\tilde{\mc{A}}_N$ is elliptic in $\Psi^{Nm(\cdot)+}$ and by \cite{FRS}, there exists a lower order operator $R\in\Psi^{Nm-1+}$ such that $\mc{A}_N=\tilde{\mc{A}}_N+R$ is essentially self-adjoint on $L^2$, elliptic, and invertible on distributions.

Define the Hilbert space $\mc{H}_N$ by
\begin{eqnarray*}
     &\mc{H}_N(F_{\mc{C}}, E)=\mc{A}_N^{-1}(L^2(F_{\mc{C}}, E))\subset\mc{D}'(F_{\mc{C}}, E),& \\
     &\croc{u, v}_{\mc{H}_N} = \croc{\mc{A}_Nu, \mc{A}_Nv}_{\mc{H}_N}.&
\end{eqnarray*}
This space contains functions that are $H^N$ regular outside of some neighbourhood of $E^*_u$ and $H^{-N}$ regular in some small neighbourhood of $E^*_u$. 
The map
\[
\mc{A}_N: \mc{H}_N(F_{\mc{C}}, E)\longrightarrow L^2(F_{\mc{C}}, E)
\]
is a unitary isomorphism, and in the $L^2$ pairing its dual is $\mc{H}^{-N}$.

By property (2) in Proposition \ref{fuggitesciocchi}, we have
\[
     \bigcap_{G, N} \mc{H}_N\subset \cminf_{E^*_u},
\]
where the intersection runs over $N\geq0$ and $G$ satisfying Proposition \ref{fuggitesciocchi}. 

On a final note we state that the definition of $\mc{H}_N$ does not cause domain problems for our vector fields. The proof of the following can again be found in  \cite{FRS}. 
\begin{prop} \label{bendef}
    Let $P\in\Psi^1$ be a pseudo-differential operator of order 1. Then  $P$ admits a unique closed extension on $\mc{H}_N$.
\end{prop}

\section{Resonances}

We start with resonances for single flows. A direct adaptation of Proposition 1.4 in \cite{DKV} shows that the Lyapunov Exponents for $\phi^A$ in some $w\in\mf{w}=\text{Crit}(\phi^A)$ are the $\{ w\alpha(A): \ \alpha\in \Delta^+_{\mc{C}}\}$, and by \cite{DR}, \cite{DR2} we deduce the following:
\begin{prop}\label{risonanze}
    The set of Pollicott-Ruelle resonances for the flow $\phi^A$ generated by $A\in\mf{a}_{++}$ is
\begin{equation*}
    \Res(A)\ = \ \left\{  \left.
    \sum_{\alpha\in w\Delta^+_{\mc{C}}\cap\Delta^+}\hspace{-10 pt} (n_{\alpha}+1) \ \alpha(A) - \hspace{-10 pt} \sum_{\alpha\in w\Delta^+_{\mc{C}}\cap\Delta^-} \hspace{ -10 pt} n_{\alpha} \ \alpha(A)\  \right| \ w\in\mf{w}, \ n_{\alpha}\in\nn \right \}. 
\end{equation*}
\end{prop}
The resonance spectrum for single flows will allow us to compute exactly the resonance spectrum for the multi-flow in the following sections. Remark that the set of resonances might constitute only a small part of the whole $\text{span}_{\mathbb{N}}\Delta^+$, as we will see in a later section where we compute Weyl laws on these sets.

The spectrum of a commuting family of operators is defined in terms of the cohomology of an exact sequence. 
We give a very brief review of the construction and collect relevant lemmas in the Appendix for convenience of the reader. A more detailed and motivated exposition can be found in \cite{GGHW}. 

Consider $\mf{a}_{\cc}= \mf{a}\otimes_{\rr} \cc$, and $\Lambda\mf{a}^*_{\cc}=\bigoplus_{\ell} \Lambda^{\ell}\mf{a}^*_{\cc}$ the exterior algebra of the dual $\mf{a}^*_{\cc}$ . We will denote $V\Lambda = V\otimes \mf{a}^*_{\cc}$, for any topological vector space $V$. If $E\rightarrow F_{\mc{C}}$ is a smooth vector bundle, we define the differential
\[
    d_{\mbf{X}-\lambda}: 
                \left\{ \begin{array}{rcl}
                \cinf(F_{\mc{C}}, E)\Lambda & \rightarrow & \cinf(F_{\mc{C}}, E)\Lambda, \\ 
                u\otimes w & \mapsto & (\mbf{X}-\lambda)u\wedge w.
                
                \end{array} \right.
\]
where $(\mbf{X}-\lambda)u\in\mf{a}^*_{\cc}$ is defined by $(\mbf{X}-\lambda)u(A)=X_Au-\lambda(A)u$. Remark that $d_{\mbf{X}-\lambda}$ extends naturally to $\cminf\Lambda$ by duality. The differential $d_{\mbf{X}-\lambda}$ is nilpotent: 
\[
d_{\mbf{X}-\lambda}d_{\mbf{X}-\lambda}=0. 
\]
Since $d_{\mbf{X}-\lambda}$ is a differential operator, it preserves the wavefront set of distribution and in particular it acts on the space $\cminf_{E^*_u}(F_{\mc{C}}, E)$ of distributions with wavefront set contained in the cotangent bundle. The interested reader can refer to \cite{mahd} for an overview on these spaces of distributions. 
These remarks allow us to consider the exact sequence
\begin{equation*}0\rightarrow\cminf_{E^*_u}\Lambda^0\xrightarrow{d_{\mbf{X}-\lambda}}\cminf_{E^*_u}\Lambda^1\xrightarrow{d_{\mbf{X}-\lambda}}...\xrightarrow{d_{\mbf{X}-\lambda}}\cminf_{E^*_u}\Lambda^{\kappa}\rightarrow0.
\end{equation*}

We define the Pollicott-Ruelle-Taylor resonances of the system as the points where the cohomology is not trivial:
\begin{equation} \label{sptaylor}
    \Res(\mc{C})=\left\{\lambda\in\mf{a}^*_{\cc} : \ 
\ker_{\cminf_{E^*_u}\Lambda}d_{\mbf{X}-\lambda}/\ran_{\cminf_{E^*_u}\Lambda}d_{\mbf{X}-\lambda} \neq 0  \right\}.
\end{equation}
Similarly, the Pollicott-Ruelle-Taylor resonant space is defined by the cohomology:
\begin{equation} \label{restaylor}
    \Res_{\mc{C}}(\lambda)=\ker_{\cminf_{E^*_u}\Lambda}d_{\mbf{X}-\lambda}/\ran_{\cminf_{E^*_u}\Lambda}d_{\mbf{X}-\lambda}.
\end{equation}
It might seem unjustified to restrict the action to this space of distributions. The fundamental reason is that Pollicott-Ruelle-Taylor resonances for hyperbolic systems always have wavefront set contained in $\cminf_{E^*_u}(F_{\mc{C}}, E)$, as in \cite{FS}, \cite{FRS}, \cite{DR}, \cite{DR2}, and the phenomenon persists for higher rank systems, as in \cite{GGHW}.

In this section we will prove the main theorem of the paper.
\begin{thm}\label{grosso}
    The Pollicott-Ruelle-Taylor spectrum is discrete, and to each element of the spectrum is associated at least one resonant state in $\cminf_{E^*_u}$:
    \begin{equation*}
        \lambda\in\Res(\mc{C}) \ \Longleftrightarrow \ \exists u\in\cminf_{E^*_u}: \ (\mbf{X}-\lambda)u=0.
    \end{equation*}
    Moreover, the set of resonances can be written explicitly as         
    \begin{eqnarray*}
            \Res(\mc{C}) &=& 
            \bigcup_{w\in\mf{w}/\mf{w}_{\mc{C}}} \left\{ \sum_{w\Delta^+_{\mc{C}}\cap\Delta^-} n_{\alpha}\alpha -\sum_{w\in\Delta^+_{\mc{C}}\cap\Delta^+} (n_{\alpha}+1) \alpha: \ n_{\alpha}\in\nn    
            \right\}
    \end{eqnarray*}
\end{thm}
The strategy of the proof is classical and consists in the reduction to a family of Hilbert spaces $\mc{H}_N$ where the above properties are relatively simple to establish. The result is then induced on the intrinsic cohomology space. 

Let us choose an escape function $G=m\log\croc{\xi}$ compatible with $A_0$ and with degrowth constant $c_0$ as in proposition \ref{fuggitesciocchi}, and let $\mc{H}_{N}$ be the anisotropic Sobolev space constructed starting from $\mc{A}_{Nm}=Op(e^{NG})+R$ as per section \ref{sobolev}. Define $\mc{D}(d_{\mbf{X}})=\{u\in\mc{H}_N: \ \mbf{X}u\in\mc{H}_N\Lambda^1\}$, which by Proposition \ref{bendef} is the unique closed domain for $d_{\mbf{X}}$, and remark that since $d_{\mbf{X}}$ is nilpotent, $d_{\mbf{X}}: \mc{D}(d_{\mbf{X}})\Lambda \rightarrow \mc{D}(d_{\mbf{X}})\Lambda$. We are then led to consider the complex 
\begin{equation}\label{Dcomplex}
    0\rightarrow \mc{D}(d_{\mbf{X}})\Lambda^0 \xrightarrow{d_{\mbf{X}-\lambda}} \mc{D}(d_{\mbf{X}})\Lambda^1 \xrightarrow{d_{\mbf{X}-\lambda}} ... \xrightarrow{d_{\mbf{X}-\lambda}} \mc{D}(d_{\mbf{X}})\Lambda^{\kappa} \xrightarrow{d_{\mbf{X}-\lambda}} 0.
\end{equation}
We define the spectrum $\sigma_{\mc{H}_N}(\mbf{X})$ of $\mbf{X}$ on $\mc{H}_N$ exactly as in \ref{sptaylor}, as the points with non-trivial cohomology.

Fix any $A_0\in\mf{a}_{++}$. Let us define the operator 
\begin{equation*}
    Q'_T(\lambda)=\int_0^Te^{-tX_{A_0}}e^{-t\lambda(A_0)}dt\in\mc{L}(\cinf(F_{\mc{C}}, E)),
\end{equation*}
and let us immediately verify that 
\begin{equation*}
    X_{A_0}Q'_T(\lambda)=Q'_T(\lambda)X_{A_0}=I-e^{-TX_{A_0}}e^{-T\lambda(A_0)}.
\end{equation*} 
Since $\phi^A$ and $\phi^B$ commute, then $e^{-tX_{A_0}}$ and $ X_B$ commute and as a result
\begin{equation*}
    X_BQ'_T(\lambda)=Q_T'(\lambda)X_B.
\end{equation*}
Next let us define $Q_T(\lambda):\mf{a}\rightarrow\mc{L}(\cinf(F_{\mc{C}}, E))$ by $Q_T(\lambda)(A)=\croc{A_0, A}Q'_T(\lambda)$
and let the divergence $\delta_{\mbf{Q}_T(\lambda)}$ be the $L^2$ formal adjoint of $d_{Q_T(\lambda)}$. 
The divergence can be in fact defined algebraically and has similar properties to that of a differential, but we don't need them here. By Lemma 3.7 of \cite{GGHW} we obtain that
\begin{equation} \label{cos1}
    d_{\mbf{X}-\lambda}\delta_{Q_T(\lambda)}+\delta_{Q_T(\lambda)}d_{\mbf{X}-\lambda}=(1-e^{-TX_{A_0}}e^{-T\lambda(A_0)})\otimes I.
\end{equation}
We give a name to the diagonal $\Lambda$-scalar operator on the right hand side: $F = (1-e^{-TX_{A_0}}e^{-T\lambda(A_0)})\otimes I$. 

The action of $e^{-TX_{A_0}}$ on $\mc{H}_N$ is conjugated to the action of $\mc{A}_Ne^{-TX_{A_0}}\mc{A}_N^{-1}$ on $L^2$.
We can then rewrite 
\begin{equation*}
    \mc{A}_Ne^{-TX_{A_0}}\mc{A}_N^{-1}= e^{-TX_{A_0}}e^{TX_{A_0}}\mc{A}_Ne^{-TX_{A_0}}\mc{A}_N^{-1} = e^{-TX_{A_0}}B_t\mc{A}_N^{-1}.
\end{equation*}
By Egorov's theorem $B_t$ is a pseudo-differential operator with principal symbol $\sigma(B_t)=\sigma(\mc{A}_N)\circ \phi^{A_0}_t=e^{NG\circ\phi^{A_0}_t}$. In particular $\sigma(B_t\mc{A}_N^{-1})=\sigma(B_t)\sigma(\mc{A}_N)^{-1}=e^{N(G\circ\phi^{A_0}_t-G)}$. Due to the choice of $G$ we know that $N(G\circ\phi^{A_0}_t-G)\leq -TNc_0$, which is to say that $B_t\mc{A}_N^{-1}\in\Psi^0$ and consequently we can decompose it into $B_t\mc{A}_N^{-1}=\tilde{K}(\lambda)+\tilde{R}(\lambda)$, with $\tilde{K}$ compact and $\|\tilde{R}\|_{\mc{L}(L^2)}\leq \limsup_{\xi\rightarrow\infty} |\sigma(B_t\mc{A}^{-1}_N)|=\limsup_{\xi\rightarrow\infty}e^{N(G\circ\phi^{A_0}_t-G)}\leq e^{-TNc_0}$.
Returning to the action on $\mc{H}_N$ we obtained the decomposition
\begin{eqnarray*}
    F(\lambda)=I+R(\lambda)+K(\lambda), \\
    R(\lambda)=e^{t\lambda(A_0)}\mc{A}_N^{-1}e^{-TX_{A_0}}\tilde{R}\mc{A}_N, \\
    K(\lambda)=e^{t\lambda(A_0)}\mc{A}_N^{-1}e^{-TX_{A_0}}\tilde{K}\mc{A}_N,
\end{eqnarray*}
with $K$ compact on $\mc{H}_{Nm}$ and $\|R\|_{\mc{L}(\mc{H}_{Nm})}=\|e^{t\lambda(A_0)}e^{-TX_{A_0}}\tilde{R}\|_{\mc{L}(L^2)}$.
With the bound $\|e^{-TX_{A_0}}\|_{\mc{L}(L^2)}\leq Ce^{TC_{L^2}(A_0)}$ we then find
\begin{equation*}
    \|R\|_{\mc{L}(\mc{H})}= \|e^{-T\lambda(A_0)}e^{-TX_{A_0}}\tilde{R}\|_{\mc{L}(L^2)} \leq Ce^{T(-\lambda(A_0)+C_{L^2}-Nc_0)}.
\end{equation*}
Where \[
C_{L^2}(A)=\limsup_{t\rightarrow+\infty} \frac{1}{t}\log \|(\phi_t^A)^*\|_{\mc{L}(L^2)}. 
\]
If $Re(\lambda(A_0))\geq C_{L^2}-Nc_0+\delta$, by taking $T$ large enough we obtain $\|R\|_{\mc{L}(\mc{H}_N)}\leq 1/2$. 
This decomposition gives us very strong information on the complex \ref{Dcomplex}.
By Lemmas 3.9, 3.10 and 3.12 of \cite{GGHW} we get that: 
\begin{itemize}
    \item By Lemma 3.9, the cohomology of \eqref{Dcomplex} is finite-dimensional;
    \item By Lemma 3.9 the projector $\Pi_0$ on the eigenvalue $0$ of $F$ on $\mc{H}\Lambda$ is bounded on $\mc{H}\Lambda$ and defines an isomorphism
    \[
    \Pi_0: \ker_{\mc{D}(d_{\mbf{X}})}d_{\mbf{X}-\lambda}/\ran_{\mc{D}(d_\mbf{X})}d_{\mbf{X}-\lambda}\longrightarrow \ker_{\ran\ \Pi_0}d_{\mbf{X}-\lambda}/\ran_{\ran \ \Pi_0}d_{\mbf{X}-\lambda} ;
    \] 
    \item By Lemma 3.10, if an element is in the spectrum it is a true eigenvalue: 
    \begin{eqnarray*}
    \lambda\in\mathcal{F}_{N, A_0, \ \delta} \ \text{and} \ \ker_{\mc{D}(d_{\mbf{X}})}d_{\mbf{X}-\lambda}/\ran_{\mc{D}(d_{\mbf{X}})}d_{\mbf{X}-\lambda}\neq 0 \  \\ \Rightarrow \ \exists u\in\mc{H}_N, u\neq0 \ \text{such that} \ d_{\mbf{X}-\lambda}u=0.
    \end{eqnarray*}
    \item By Lemma 3.12, the spectrum $\sigma_{\mc{H}_N}(X)$ is discrete.
\end{itemize}

What we now have to do is to push these properties from the complex on $\mc{H}_N$ to the intrinsic complex on $\cminf_{E^*_u}$, but that is a consequence of the following proposition, whose proof follows verbatim the proof of Proposition 4.10 in \cite{GGHW}. 
\begin{prop}
    For all $\lambda\in\mc{F}_{N, A_0}$ there is an isomorphism
    \begin{equation*}
                \quad \ker_{\mc{D}(d_{\mbf{X}})\Lambda^j}d_{\mbf{X}-\lambda}/\ran_{\mc{D}(d_{\mbf{X}})\Lambda^j}d_{\mbf{X}-\lambda} \simeq \ker_{\cminf_{E^*_u}\Lambda^j}d_{\mbf{X}-\lambda}/\ran_{\cminf_{E^*_u}\Lambda^j}d_{\mbf{X}-\lambda} 
    \end{equation*}
\end{prop}

This shows that the spectra we discovered on the spaces $\mc{H}_{N}$ are intrinsic to the flow. As a corollary, the Taylor spectrum is discrete and pure point, i.e. made of eigenvalues with finite-dimensional eigenspaces. Moreover, the resonant states are distributions in $\cminf_{E^*_u}$, so that the first part of Theorem \ref{grosso} is proven. 

We have thus found many regularity properties of the spectrum in an abstract way.

This last result is key because it allows very easily to identify multi-resonances with the one-dimensional resonances we had found in the introduction. 
\begin{prop} \label{explicit}
    The Resonances of the flow are  
        \begin{eqnarray*}
            \Res_{\phi} &=& 
            \bigcup_{w\in\mf{w}/\mf{w}_{\mc{C}}} \left\{ \sum_{w\Delta^+_{\mc{C}}\cap\Delta^-} n_{\alpha}\alpha -\sum_{w\Delta^+_{\mc{C}}\cap\Delta^+} (n_{\alpha}+1) \alpha: \ n_{\alpha}\in\nn    
            \right\} \\
            &\subset& \left\{ -\sum_{\alpha\in\Pi} n_{\alpha}\alpha: \ n_{\alpha}\in\nn \right\} 
        \end{eqnarray*}
\end{prop}
\begin{proof}
    Suppose $\lambda\in \Res(\mc{C})$. Since $\text{Span}_{\cc}\ \Delta^+=\mf{a}^*_{\cc}$, there exist coefficients $x_{\alpha}\in\cc$ such that $\lambda = \sum_{\alpha\in\Delta^+} z_{\alpha}\alpha$. Remark that these coefficients are not necessarily unique, as $\Delta^+$ is not free. 

    By the above results, there exists $u\in\cminf$ such that $(X+\lambda)u=0$, that is $X_Au=-\lambda(A)u$ for all $A\in\mf{a}_{++}$. In particular, for any $A$, $\lambda(A)$ is a resonance for the flow $\phi^A$, therefore as seen in Proposition \ref{risonanze} this implies that 
    \[
    \lambda(A) = \sum_{w(A)\Delta^+_{\mc{C}}\cap\Delta^-} n_{\alpha}(A)\alpha(A) -\sum_{w(A)\in\Delta^+_{\mc{C}}\cap\Delta^+} (n_{\alpha}(A)+1) \alpha(A),
    \]
    for some $n_{\alpha}(A)\in\mathbb{Z}_{\geq0}$ and some $w(A)\in\mf{w}$. But remark that $\lambda$ is linear in $A$, and so are all the $\alpha$. Given that $\mf{a}^+$ is open, the $n_{\alpha}(A)$ are constant with respect to $A$, and as a result $\lambda=-\sum_{w(A)\Delta^+_{\mc{C}}} n_{\alpha}\alpha$. 
\end{proof}
This concludes the proof of Theorem \ref{grosso}.

\section{Resonances for Grassmanians}

The set of resonances for a general flow is complicated, and explicit computations for low-dimensional cases show that it is not even convex, in the sense that in general $\Res(\mc{C}) \subsetneq \mathbb{Z}^n\cap \text{conv}(\Res(\mc{C}))$. Explicit computations can be carried out in some special cases. In this section we give a more explicit expression of the resonant set for the maximal flag manifold, corresponding to the Borel subgroup, and for projective planes. 
We then compute some multi-dimensional Weyl laws.

\begin{cor}
    Suppose the parabolic subgroup $P$ is minimal, that is $P=B$ the Borel subgroup, or equivalently that $\mc{C}=\mf{a}_{++}$. Then 
    \begin{equation*}
        \Res_{\phi}= \left\{ -\sum_{\alpha\in\Delta^+} n_{\alpha}\alpha: \ n_{\alpha}\in\nn \right\} = -\text{Span}_{\nn}\  \Pi .
    \end{equation*}
\end{cor}
\begin{proof}
    Clearly $\Res_{\phi}\subset -Span_{\nn} \ \Pi$. On the other hand, consider $w\in\mf{w}$ the maximal element, so that $w\Delta^+=\Delta^-$. Then since $\Delta^+_{\mc{C}}=\Delta^+$, we find
    \begin{eqnarray*}
        \Res_{\phi} &\supset& \Res_{\phi}(wP) = \\
         &=& \left\{ -\sum_{\alpha\in w\Delta^+_{\mc{C}}\cap\Delta^+} (n_{\alpha}+1) \ |\alpha(X)|-\sum_{\alpha\in w\Delta^+_{\mc{C}}\cap\Delta^-} n_{\alpha} \ |\alpha(X)|:  \ n_{\alpha}\in\nn \ \right\}, \\   
        &=& \left\{ -\sum_{\alpha\in\Delta^+_{\mc{C}}} n_{\alpha}\alpha: \ n_{\alpha}\in\nn \right\} = -\text{Span}_{\nn} \ \Pi.
    \end{eqnarray*}
\end{proof}

We now restrict ourselves to the special linear group $G=SL_n(\rr)$. We take the canonical simple system $\Pi= \{ e^*_i-e^*_{i+1}:\ i=0, ..., n-1\}=\{\alpha_{i, i+1}: \ i=0, ..., n-1\}$. Since $\Res_{\phi}\subset -Span_{\nn} \ \Pi$, we can identify any element $R\in \Res_{\phi}$ by its coefficients in the $\Pi$ basis: $R=-\sum_i c_i\alpha_{i, i+1}$. 

\begin{cor}\label{resonanzepargrande}
    Suppose $G=SL_n(\rr)$, and $P$ is the maximal parabolic subgroup. Then the set of resonances is
    \begin{eqnarray*}
        \Res(\mc{C}) = \{ -\sum_i c_i\alpha_i: && \exists k\in\{0, ..., n-1\} \  \textit{such that} \\
        && 0<c_1<c_2< ...<c_k, \\
        && c_{k+1}\geq c_{k+2}\geq ...\geq c_{n-1} \}.
    \end{eqnarray*}
\end{cor}
\begin{proof}
    First of all, see that $\mc{C}=\{diag(h_1, ..., h_n): \ h_1=h_2=...=h_{n-1}>0, \ h_n=-nh_1\}$, and 
    \[
    \Delta^+_{\mc{C}} = \{ \alpha_{k,n}: k=1, ..., n-1\}.
    \]
    Recall that $\mf{w}\simeq S_n$ and $\mf{w}_{\mc{C}}\simeq S_{n-1}$ acting on the first $n-1$ coordinates. Therefore, 
    \[
    \{ (k,n) : \ k=0, ..., n-1\}\subset S_n
    \]
    is a complete set of representatives for $\mf{w}/\mf{w}_{\mc{C}}$. Now we can see that under the action of $(k, n)$:
    \begin{eqnarray*}
    i<k: \ \alpha_{i, n} &\mapsto& \alpha_{i, k} \ \in\Delta^+, \\
        \alpha_{k, n} &\mapsto& \alpha_{n, k} \ \in\Delta^-, \\
    i>k: \    \alpha_{i, n} &\mapsto& \alpha_{i, k} \ \in\Delta^-.
    \end{eqnarray*}
    We then find that
    \[
    \Res_{\phi}((k, n))= \left\{ -\sum_{i<k} (c_j+1)\alpha_{i, k} - c_k\alpha_{i, k} -\sum_{i>k} c_i \alpha_{i, k} : \ c_i, c_k\in\nn\right\},
    \]
    So by developing each root as sum of simple roots:
    \begin{eqnarray*}
        i<k: \ \alpha_{i, k} &=& \alpha_{i, i+1}+...+\alpha_{k-1, k} \\
        k: \ \alpha_{k, n} &=& \alpha_{k, k+1}+ ...+\alpha_{n-1, n} \\
        i>k: \ \alpha_{k, i} &=&  \alpha_{k, k+1}+...+\alpha_{i-1, i}
    \end{eqnarray*}
    So we get 
    \begin{eqnarray*}
        \Res_{\phi}((i, n)) = -\{  \\
        (c_1+1) &\alpha_{1, 2}& \\
        +(c_1+c_2+2) &\alpha_{2, 3}& \\
        +... && \\
        +(c_1+...+c_{k-1}+(k-1))&\alpha_{k-1, k}& \\
        +(c_k+c_{k+1}+...+c_{n-1})&\alpha_{k, k+1}& \\
        +(c_k+c_{k+2}+...+c_{n-1})&\alpha_{i+1, i+2}& \\
        +... && \\
        +(c_k+c_{n-1})&\alpha_{n-2, n-1}& \\
        +c_k&\alpha_{n-1,n}:& \\
         c_k\in\nn \ \forall k =1, ..., n \}.
    \end{eqnarray*}
    We can then rename
    \begin{eqnarray*}
        C_i&=&c_1+...+c_i+i     \qquad  i<k \\
        C_i&=&c_k+c_{i+1}+c_{i+2}+...+c_{n-1}   \qquad k\leq i <n
    \end{eqnarray*}
    And we see that the $C_j$ are allowed any value, with the constraint being
    \begin{eqnarray*}
        0<C_1<C_2...< C_{i-1} \\
        C_{i}\geq C_{i+1}\geq...\geq C_{n-1}\geq0.
    \end{eqnarray*}
    The complete resonant set is 
    \[
    \Res_{\phi} = \bigcup_{1\leq k \leq n-1} \Res_{\phi}((k, n)),
    \]
    so we get the anticipated description.
\end{proof}
One can see that the spectrum respects basic monotonicity conditions. 
\begin{lemma}
    Suppose $\mc{C}\subset \mc{C}'$, then $\Res(\mc{C})\subset \Res(\mc{C'}).$
\end{lemma}
\begin{proof}
    We have 
    \[
    \mc{C}\subset\mc{C'} \Rightarrow \Delta^0_{\mc{C}}\supset \Delta^0_{\mc{C'}}\Rightarrow \Delta^+_{\mc{C}}\subset \Delta^+_{\mc{C'}}
    \]
    hence the proposition follows from Proposition \ref{explicit}.
\end{proof}
Unfortunately the resonance spectrum does not identify uniquely the flag manifold, even if one of the chambers contains the other. For example, if $G=SL(3, \rr)$, by taking $\mc{C}=\{ h_1=h_2=-2h_3\}\subset \mf{a}$ and $\mc{C}'=\mf{a}^{++}$ we have that $\mc{C}\subset \mc{C'}$ and 
\begin{equation}\label{cringe}
\Res(\mc{C})=-\text{Span}_{\nn}\{ e_1-e_2, e_2-e_3\}=\Res(\mc{C'}),
\end{equation}
but clearly $F_{\mc{C}}=\mathbb{P}_2\neq F_{\mc{C'}}$ as they have different dimensions.

We now compute some Weyl laws. Clearly, in a multidimensional setting the Weyl law depends on which metric to use on $\mf{a}^*$. The author found it more natural to state the results in terms of the classical euclidean distance, with $\Pi$ taken as an orthonormal basis. The laws for the $L^{\infty}$ and the $L^1$ distances are derived identically. From now on, we will denote \[
N(\lambda) = \#\{ \alpha\in\Res(\mc{C}): \ |\alpha|_2<\lambda\}
\]

For convenience we will take $G=SL(n+1, \rr)$, so that $dim\ \mf{a}=n$ and the statements will be simpler. 
\begin{prop}{Weyl law}
    Suppose $\mc{C}=\{h_1=...=h_n=-nh_{n+1}\}$, then the Weyl law for resonance reads as
    \[
    N(\lambda)=|\Res(\mc{C})\cap B_{\mf{a}}(0, \lambda)|=\frac{\omega_n}{2\cdot n!}\lambda^{n/2}+O(\lambda^{(n-1)/2})
    \]
\end{prop}
\begin{proof}
    By using \ref{resonanzepargrande}, evaluating $N(\lambda)$ is equivalent to counting elements $(x_1, ..., x_n)\in\nn^n$ such that there exists $k$ such that 
    \begin{equation}\label{gln}
        x_1<...<x_k \quad \textit{and} \quad x_{k+1}\geq...\geq x_n
    \end{equation}
     and moreover $x_1^2+...+x_n^2< \lambda^2$.
    Suppose we are given $n$ elements $x_1<x_2<...<x_n$: we see that there are exactly $2^{n-1}$ permutations $\sigma\in S_n$ such that $(x_{\sigma(1)}, ..., x_{\sigma(n)})$ satisfy \eqref{gln}. Indeed, choosing such a permutation $\sigma$ is equivalent to choosing the subset of elements $A = \{ x_{i_1}, ..., x_{i_k}\}$ that make up the "ascending" chain in \eqref{gln}, once this is done the ordering they appear in and the ordering of the remaining elements are forced by the fact that they are all different. The only repetition comes from the fact that each $\sigma$ is counted exactly twice, indeed if $x_{\sigma(k)}>x_{\sigma(k+1)}$, then the choices of $A = \{ x_{\sigma(1)}, ..., x_{\sigma(k-1)}\}$ and $A = \{x_{\sigma(1)}, ..., x_{\sigma(k)}\}$ both lead to $\sigma$, and similarly if $x_{\sigma(k)}<x_{\sigma(k+1)}$. Now, see that $|X|=|\nn^n\cap B(0, \lambda)|=w_n\lambda^{n/2}/2^n+O(\lambda^{(n-1)/2})$. Take any $(x; 1, ..., x_n)\in X$, then $(x_{\sigma(1)}, ..., x_{\sigma(n)})\in X$ for any permutation $\sigma\in S_n$, so $|\{ (x_1, ..., x_n)\in X: \ x_1<x_2<...<x_n\}|=|X|/n! +O(\lambda^{(n-1)/2})$, as elements with repeated entries are at most $O(\lambda^{(n-1)/2})$. By the first part of this proof, we conclude that 
    \begin{eqnarray*}
            N(\lambda) &=& 2^{n-1}|\{ (x_1, ..., x_n)\in X: \ x_1<x_2<...<x_n\}|+O(\lambda^{n-1}) \\
            &=& \frac{w_n}{2\cdot n!} \lambda^{n/2}+O(\lambda^{(n-1)/2})
    \end{eqnarray*}

\end{proof}

Consider now a Grassmanian, so that $\mc{C}=\mc{C}_k=\{ h_1=...=h_k, \ h_{k+1}=...=h_n, \ \sum h_i=0\}$. Denote
\begin{eqnarray*}
    K = \{1, ..., k\} \quad K^c=\{k+1, ..., n\}.
\end{eqnarray*}
Now, for any $\ell\leq \min\{k, n-k\}$, for any $A=\{a_1\leq ...\leq a_{\ell}\}\subset K$, $B=\{b_1\leq  ...\leq b_{\ell}\}\subset K^c$, define $(A, B)=(a_1b_1)(a_2b_2)...(a_nb_n)\in S_n$, that is: $(A, B)$ is the permutation sending the elements of $A$ to the elements of $B$, preserving the ordering. Then we have that the $(A, B)$ generate the quotiented Weyl group, that is
\[
\mf{w}/\mf{w}_{\mc{C}}= \{ (A, B)\mf{w}_{\mc{C}}: \ A\subset K, \ B\subset K^c, \ |A|=|B|=\ell\leq \min\{k, n-k\} \} = W
\]
Indeed, first we remark that by definition $W\subset \mf{w}/\mf{w}_{\mc{C}}$. See that the classes of $W$ are all distinct, as if $A\neq A'$ and $|A|=|A'|$, then there exists some element $i\in K$ such that $(A, B)i \in K^c$ but $(A', B')\in K$, and since $\mf{w}_{\mc{C}}$ stabilises $K$ and $K^c$, we get that $(A, B)\mf{w}_{\mc{C}}\neq (A', B')\mf{w}_{\mc{C}}$. Therefore, $K\subset \mf{w}/\mf{w}_{\mc{C}}$ injectively. 
Then we remark that $\mf{w}/\mf{w}_{\mc{C}}\simeq S_n/(S_k\times S_{n-k})$, so $|\mf{w}/\mf{w}_{\mc{C}}|=\frac{n!}{k!(n-k)!}= {n \choose k}$. Now, notice that the choice of any set $C$ of $k$ random elements in $\{1, ...,n\}$ is equivalent to the choice of $A=K \setminus C$ and $B = K^c\cap C$, hence $W={n\choose k} = \mf{w}/\mf{w}_{\mc{C}}$. \\
Now we can compute the action of $(A, B)$ on $\Delta^+_{\mc{C}}$. 
\begin{eqnarray*}
    (A, B): \ &i\notin A, j\notin B& \quad \alpha_{i, j}\mapsto \alpha_{i, j} \\
    &i=A(h) \in A, j\notin B& \quad \alpha_{i, j}\mapsto \alpha_{B(h), j} \\
    &i\notin A, j=B(h)\in B& \quad \alpha_{i, j}\mapsto \alpha_{i, A(h)} \\
    &i=A(h) \in A, j=B(h)\in B& \quad \alpha_{i, j}\mapsto \alpha_{B(i), B(j)} \\
\end{eqnarray*}
and one can see that the whole action can be summarised in 
\begin{eqnarray*}
    (A, B)\Delta^+_{\mc{C}}
    = \{ \alpha_{i, j}: \ && \  i\in (K\setminus A), \ j\in (K^c\setminus B) \\
    && or \quad i\in B, \ j\in (K^c\setminus B) \\
    && or \quad  i\in (K\setminus A), \ j\in A \\
    && or \quad i\in A, j\in B \qquad \qquad \}
\end{eqnarray*}
This allows for explicit computation for some specific cases. For instance, we have the following
\begin{prop}\label{parabolicosimm}
    Let $G=SL(n+1, \rr)$. If $n+1=2k$, set $\mc{C}=\{ h_1=...=h_k, \ h_{k+1}=...=h_{n+1}, \ h_1+h_{k+1}=0\}$, and if $n+1=2k+1$ let $\mc{C}=\{ h_1=...=h_{k+1}, \ h_{k+2}=...=h_{n+1}, \ (k+1)h_1+kh_{k+1}=0\}$. Then, we have that 
    \[
    N_{\mc{C}}(\lambda) = \frac{\omega_n}{2^n}\lambda^{n/2}+O(\lambda^{(n-1)/2})
    \]
\end{prop}
\begin{proof}
    Suppose $n=2k$.
    Consider $A=\{ 2, 4, ..., 2\lfloor\frac{k}{2}\rfloor\}$ and $B=\{2\lfloor (k+1)/2\rfloor +1, 2\lfloor (k+1)/2\rfloor +3, ..., 4\lfloor(k+1)/2\rfloor -1\}$. Then see that, for $i=1, ..., n-1$ we have:
    \begin{eqnarray*}
        i<k &:& \ \textit{either} \quad i\in A, \ i+1\in K\setminus A \quad \textit{or} \quad i\in K\setminus A,\  i+1\in A \\
        i=k&:& \ \textit{if k is odd} \quad k\in K\setminus A, \ k+1\in K^c\setminus B \\
        &&\  \textit{if k is even} \quad k\in A, \ k+1\in B \\
        i>k &:&\ \textit{either} \quad i\in B, \ i+1\in K\setminus B \quad \textit{or} \quad i\in K\setminus B, \ i+1\in B
    \end{eqnarray*}
    Hence, all the simple roots are in $(A, B)\Delta^+_{\mc{C}}$. Since we are only interested in the first term of the asymptotic expansion of $\lambda$, we get that 
    \[
    N(\lambda)=|\text{span}_{\nn}\Pi \ \cap B(0, \lambda)|+O(\lambda^{(n-1)/2}) = \frac{\omega_n}{2^n}\lambda^{n/2}+O(\lambda^{(n-1)/2})
    \]
    The case for $n+1=2k+1$ is analogous. 
\end{proof}

There are several question one might be tempted to ask at this point, for instance, do the spectra determine the flag manifolds uniquely? The answer is clearly no, as remarked in \eqref{cringe}. Nevertheless, one might ask if this is the case among a restricted subset of flag manifolds, such as the set of minimal flags i.e. Grassmanians. Partial resulults can be proven, such as the fact that the "central" Grassmanian with $k= \lfloor n/2\rfloor$ has maximal Weyl constant among Grassmanians of dimension $n$. This can be proven by showing that if $k\neq \lfloor n/2\rfloor$, then any element $a=(a_1, ..., a_n)\in \Res(\mc{C}_k)$ must have three indices $(a_{j-1}, a_j, a_{j+1})$ such that $a_j\leq a_{j-1}+a_{j+1}$, and then counting that the set sequences $(a_1, ..., a_n)$ such that the previous inequality is false for any choice of $j$ has volume $\sim c\lambda^{n/2}$ and is therefore not negligible. In a similar manner, as long as $1<k<n-1$ then $\Res(\mc{C}_k)$ allows non-monotonic sequences that increase the Weyl constant strictly from the one of $F_{n-1}$. 
The author suspects that this might extend to a strict inequality of Weyl constants for Grassmanians of fixed dimension. 


\appendix

\section{Proof of existence of the invariant neighbourhood}\label{intorno}

We give here the complete proof of proposition \ref{intornofisso}. The proof is in fact very similar to the one in \cite{DR}, with the additional difficulty that we are dealing with a family of flows instead of with a single one. It is unclear if such a proof could exist without the (almost) global coordinate charts available in our situation. 

\begin{prop}
     Let $V_0^s\supset\Sigma_s$ be an open set. Then there exists $V^s\subset V_0^s$ containing $\Sigma_s$ that is positively flow-invariant for all $A\in\mf{a}_{++}$. 
\end{prop}
\begin{proof}
Let $\{w_1, ..., w_p\}=\mf{w}P$ be the critical points of the flow, ordered so that $W^u(w_i)\cap W^s(w_j)\neq \varnothing \ \Rightarrow i<j$, see \cite{DR} as to why this is possible. We denote $\pi_F: T^*F_{\mc{C}}\rightarrow F_{\mc{C}}$ the canonical projection.

Let us construct the set $\tilde{V}^s$ progressively starting from $w_p$, the only element whose stable manifold has full dimension. In the $\gamma_{w_p}$ parametrization, since $V_0^s$ is open, there exists a small $\ve_p$ such that $U_p=(-\ve_p,\ve_p)^{n} \subset \gamma_{w_p}\pi_F(V_0^s)$, and clearly $U_p$ is positively flow-invariant for all $X\in\mf{a}_{++}$.

Suppose then that we have constructed $U_{k+1}\supset \bigcup_{j\geq k+1} W^u(w_j)$, and let us take the coordinates of $w_k$, with the separation of coordinate indices $\Delta^{\pm}=w_k\Delta^+_{\mc{C}}\cap\Delta^{\pm}$ and $x_{s, u}=(x_{\alpha})_{\Delta^{\pm}}$.
There exists $R>0$ such that if $|x_u|\geq R$ then $(0, x_u)\in U_{k+1}$, and therefore there exists $\eta>0$ such that if $|x_s|\leq \eta,\  |x_u|=R$ then $(x_s, x_u)\in U_{k+1}$.
As $V_0^s$ is open, there exists $\eta'<\eta$ such that $U^{(k)}=[-\eta', \eta']^{\Delta^-}\times [-R, R]^{\Delta^+}\subset \pi_F\tilde{V}_0^s$.
But then let us remark that $U_k=U^{(k)}\cup U_{k+1}$ is positively flow-invariant for all $X\in \mf{a}_{++}$ and furthermore $U_k\subset \pi_F(V^s_0)$.
In fact, let $(x_s, x_u) \in U^{(k)}$ and $X\in \mf{a}_{++}$: there exists a certain $t_0$ such that $|\phi^X_{t_0}(x_u)|=R$.
Since $|\phi^X_t(x_s)|$ is decreasing, for all $t\in[0, t_0]$ we have $(x_s, x_u)\in[-\eta', \eta']^{\Delta^-}\times [-R, R]^{\Delta^+}=U^{(k)}\subset U_k$, then at $t_0$ we have $\phi^X_{t_0}(x_s, x_u)\in U_{k+1}$ and by flow-invariance of $U_{k+1}$ we find that $\forall t>0$ we have $\phi^X_{t_0+t}(x_s, x_u)\in U_{k+1}$.
In particular, $\forall t\geq 0: \ \phi^A_t(x_s, x_u)\in U_k$, which proves the flow-invariance. 

The construction of the lift to the cotangent bundle is simple from here, we prove the existence of positively flow-invariant conic subsets in the coordinate charts and glue them. 

Indeed, we proceed inductively starting with $\mc{U}_p=T^*U_p$, which is clearly positively flow invariant. At each step we fix a critical point $w_k$ with the associated coordinates, we are given $\mc{U}_{k+1}$ positively flow invariant in $S^*F_{\mc{C}}$, and the set $U_k\subset F_{\mc{C}}$ constructed above. Consider $S^*F_{\mc{C}}$ as the unit co-sphere bundle in these coordinates. Just as above, there exist $R, \eta$ such that $|x_s|<\eta,\ |x_u|=R$ implies $(x_s, x_u)\in U_{k+1}$. In the same way, there exists $\varepsilon>0$ such that $|x_s|<\eta, \ |x_u|=R, \ |\xi_u|/|\xi_s|<\varepsilon$ implies $(x, \xi)\in \mc{U}_{k+1}$. Then if we take 
\[
\mc{U}_k = \mc{U}_{k+1}\cup U_k\times \{ |\xi_u|<\varepsilon' |\xi_s|\}\subset S^*F_{\mc{C}},
\]
we find that $\mc{U}_k$ is positively flow invariant, and clearly $\mc{U}_k\supset \Sigma_s\cap \overline{W^u(w_k)}$ since $W^u(w_k)\cap \Sigma_s=\{ x_s=0, \ \xi_u=0\}$. By taking $\varepsilon'>0$ small enough, we ensure that $\mc{U}_k\subset \tilde{V}^s$. Going up to $p=1$ we complete the proof.
\end{proof}

\section{Resolution of the algebraic condition}\label{algebra}

We give here the complete classification of \emph{regular pairs} $(\mc{C}, X)$ for which our analysis of the flow is valid. For convenience, take some $H\in\mc{C}$ such that $\Delta^0_{\mc{C}}=\Delta^0_H=H^{\perp}\cap \Delta$. 
Our goal is to find which pairs $(X, H)$ satisfy the following property: 
\begin{equation} \label{cabbo}
    \forall\alpha\in\Delta: \ (\alpha(X)=0)\ \Rightarrow \ (\forall w\in\mf{w}: \ \alpha(wH)=0) .
\end{equation} 
It is clear that the condition is invariant under the action of $\mf{w}$:
$\forall w_1, w_2\in\mf{w}$we have that $(X, H)$ satisfies \eqref{cabbo} if and only if $(w_1X, w_2H)$ does.
Therefore if we fix a positive Weyl chamber $\mf{a}_{++}$ we can assume $X, H\in \overline{\mf{a}_{++}}$.
Recall that an element $Y\in\mf{a}$  is called regular if $\alpha(Y)\neq0 \ \forall\alpha\in\Delta$, and the set of regular elements is $\mf{w}\mf{a}_{++}$. 
If $X$ is regular, the antecedent in \eqref{cabbo} is never satisfied, so $(X, H)$ is a solution. 
If $H$ is regular, then the antecedent in \eqref{cabbo} can never be satisfied, so $X$ must be regular.
We want to go all the way and find all non-regular pairs.
Given the formulation of the proposition, it is natural to look for an $X$ such that solutions $H$ exist: we will denote $\mathcal{H}(X)$ the set of such $H$.
By inspecting \eqref{cabbo} we see that if $\mf{w}$ is transitive on the roots, there exists no interesting non-regular $X$.
For the sake of the argument, if $\mf{w}$ were to fix all the roots,  then for all $X\neq0$, $\mc{H}(X)\neq\{0\}$ as $X$ itself is a solution. What we are trying to quantify is therefore this $\ll \textit{lack of transitivity} \gg$ . 
Let us now recall the existence of a scalar product $\croc{\ , }$ on $\mf{a}, \mf{a}'$ such that $\mf{w}$  is generated by the orthogonal reflections $s_{\alpha}$ with respect to the kernels of $\croc{\alpha, \cdot}_{\mf{a}'}$.
Let us define
\begin{equation*} 
    R(X)=\Delta^0_X\cup \{\alpha\in\Delta: \ \exists\beta\in\Delta(X) \ \textrm{such that} \ \croc{\alpha, \beta}\neq0 \}=\Delta\setminus (\Delta^0_X)^{\perp}.
\end{equation*}
Our main proposition will be
\begin{prop} \label{suino}
    Let $X\in\mf{a}$.
    Then $\mathcal{H}(X)=\bigcap_{\alpha\in R(X)} \ker(\alpha)$. 
\end{prop}
\begin{proof}
    Let us start with two preliminary lemmas.
    \begin{lemma} \label{suinoprimo}
        Suppose $\croc{\alpha, \beta}\neq0$. Then $ker(\alpha)\cap \ker(\beta)=ker(s_{\beta}\alpha)\cap \ker(\beta)$.
    \end{lemma}
    \begin{proof}
        Take an orthonormal basis of $\mf{a}'$ such that $\beta=\beta_1e_1$ and $\alpha=\alpha_1e_1+\alpha'$, with $\alpha'\perp\beta$, and by non-orthogonality $\alpha_1\neq0$. The proof follows from a direct computation. 
    \end{proof}
    \begin{lemma}
        For all $\beta\in\Delta$, we have  
        \[ \bigcap_{\gamma_\in R(X)} \ker(s_{\beta}\gamma) = \bigcap_{\gamma_\in R(X)} \ker(\gamma) .\] 
        In particular, the Weyl group preserves $\bigcap_{\gamma_\in R(X)} \ker(\gamma)$.
    \end{lemma}
    \begin{proof}
        Suppose first that $\beta\in\Delta(X)^{\perp}$.
        Then $s_{\beta}\gamma=\gamma$ for all $\gamma\in\Delta(X)$, and as $s_{\beta}$ is an isometry, it preserves $\Delta(X)^{\perp}$, in particular $s_{\beta}(R(X))=s_{\beta}(\Delta\setminus\Delta(X)^{\perp})=\Delta\setminus \Delta(X)^{\perp}=R(X)$.
        Suppose now that $\beta\in R(X)$. Then
        \begin{eqnarray*}
            \bigcap_{\gamma\in R(X)} \ker(s_{\beta}\gamma) &=& \bigcap _{\gamma\in R(X)} \ker(s_{\beta}\gamma)\cap \ker(\beta) \\
            &=& 
        \bigcap_{\gamma\in R(X), \croc{\beta, \gamma}=0} \ker(s_{\beta}\gamma) \quad \cap \bigcap_{\gamma\in R(X), \croc{\beta, \gamma}\neq0} \ker(s_{\beta}\gamma)\cap \ker(\beta) \\ &=& 
        \bigcap_{\gamma\in R(X), \croc{\beta, \gamma}=0} \ker(\gamma)  \quad \cap \bigcap_{\gamma\in R(X), \croc{\beta, \gamma}\neq0} \ker(\gamma)\cap \ker(\beta) \\ &=&
        \bigcap_{\gamma\in R(X)} \ker(\gamma).
    \end{eqnarray*}

        As the $s_{\beta}$ generate the Weyl group, the lemma follows.
    \end{proof}
    Let us continue with the proof of Proposition \ref{suino}.

    Suppose $H\in\bigcap_{\gamma\in R(X)} \ker(\gamma)$.
    If $w\in\mf{w}$ and $\alpha\in\Delta(X)$, then 
    \begin{equation*}
        H\in \bigcap_{\gamma\in R(X)} \ker(\gamma)=\bigcap_{\gamma\in R(X)} \ker(w^{-1}\gamma)\subset \ker(w^{-1}\alpha),
    \end{equation*} that is to say, $\alpha(wH)=0$.
    We have shown that $H\in\mathcal{H}(X)$.

    Now suppose that $H\in\mathcal{H}(X)$. Clearly $H\in\bigcap_{\alpha\in\Delta(X)} \ker(\gamma)$.
    Let us take any $\beta\in R(X)\setminus \Delta(X)$, in particular there exists $\gamma\in\Delta(X)$ such that $\croc{\beta, \gamma}\neq0$.
    Then, by lemma \ref{suinoprimo},we have that $H\in \ker(\gamma)\cap \ker(s_{\beta}\gamma)=ker(\gamma)\cap \ker(\beta)$. 
    We therefore have $H\in\bigcap_{\gamma\in R(X)} \ker(\gamma)$.
\end{proof}
According to this result, to have non-trivial solutions, the root system must have strong orthogonality properties.
We can see, in fact, that no irreducible root system has a non-trivial solution.

\begin{prop}\label{radirred}
    Suppose that the root system $\Delta$ is irreducible, then for all non-regular $X$ we have $\mathcal{H}(X)=\{0\}$.
\end{prop}
\begin{proof}
Every irreducible root system is part of a class of type $A_n, B_n$, $C_n, D_n$ or one of the sporadic systems $E_6, E_7, E_8, F_4, G_2$. 
For the sake of brevity, we only give the proof for the $A_n$.
Systems of type $A_n$: $A_n$ is realized as a subset of $E=\{ x\in\rr^{n+1}: \sum_i x_i=0\}$, it is the set of elements with one coordinate equal to -1 and one equal to 1.

The standard set of simple roots is $\{e_i-e_{i+1}: i\leq n\}$, if $n\geq1$.
Suppose then $n\geq 1$, and $\Delta(X)\neq\varnothing$, that is to say  $(e_i-e_j)(X)=0$ for some pair $(i, j)$.
As for all $k\neq i$: $\croc{e_i-e_j, e_i-e_k}\neq0$, we have $R(X)\supset \{e_i-e_k: k\neq i\}$.
These are $n$ linearly independent elements, in particular their kernels have intersection $0$, so $\mathcal{H}(X)=\{0\}$.
\end{proof}

We are finally able to give a complete classification. Suppose that the root system $\Delta$ decomposes into $\Delta= \Delta^{(1)}\cup...\cup\Delta^{(m)}$, where the $\Delta^{(j)}$ are all irreducible and pairwise orthogonal.
Each $\Delta^{(j)}$ generates a subspace $V^{(j)}\subset\mf{a}^*$, and the $V^{(j)}$ form a decomposition of $\mf{a}^*$ into orthogonal subspaces.
To each $V^{(j)}$ is associated a subspace $\mf{a}^{(j)}=(\bigcap_{\Delta^{(j)}} \ker(\alpha))^{\perp}\subset\mf{a}$, which thus form an orthogonal decomposition of $\mf{a}$.
Let us remark that each $\Delta^{(j)}$ is a root system associated with $\mf{a}^{(j)}$.
Next we note that this also induces a decomposition of the Lie algebra $\mfg$: for each $j$ let us define $\mf{g}^{(j)}=\mf{a}^{(j)}+\sum_{\alpha\in\Delta^{(j)}} \mfg_{\alpha}$.
We have $\mfg=\mf{m}+\mf{a}+\sum_{\alpha\in\Delta}\mfg_{\alpha}=\mf{m}+\sum_{j} \mfg^{(j)}$. We can see that $\mf{a}^{(j)}$ normalizes $\mfg^{(j)}$, and for each $\alpha, \beta\in\Delta^{(j)}$ if $\alpha+\beta\in\Delta$, then $\alpha+\beta\in\Delta^{(j)}$, so $\mfg^{(j)}$ is a subalgebra of $\mfg$, as $[\mfg_{\alpha}, \mfg_{\beta}]\subset\mfg_{\alpha+\beta}$.
On the other hand, each $\mf{a}^{(i)}$ commutes with $\mfg^{(j)}$ for $i\neq j$.
If $\alpha\in\Delta^{(i)}$ and $\beta\in\Delta^{(j)}$ then $\alpha+\beta\notin\Delta$: otherwise, as $\croc{\alpha, \beta}=0$ we have that $\croc{\alpha, \alpha+\beta}\neq0$ and $\croc{\beta, \alpha+\beta}\neq0$ and as the $\Delta^{(k)}$ form an orthogonal decomposition of $\Delta$, this would imply $\alpha+\beta\in \Delta^{(i)}\cap\Delta^{(j)}=\varnothing$.
Therefore each $\mfg^{(j)}$ is an ideal of $\mfg$. 
Finally, let us remark that if our original Lie algebra is equipped with a Cartan automorphism, and $\mfg=\mf{k}+\mf{a}+\mf{n}$ is the Iwasawa decomposition, we can see that each $\mfg^{(j)}$ is fixed by the Cartan automorphism: in fact $\mf{a}^{(j)}\subset \mf{s}$ is the eigenspace of $-1$ for $\theta$,  and for each $\alpha\in\Delta^{(j)}$ we have  $-\alpha\in\Delta^{(j)}$ and $\theta\mfg_{\alpha}=\mfg_{-\alpha}\subset\mf{g}^{(j)}$.
Recall that a semisimple Lie algebra $\mf{p}$ is simple if and only if its root system is simple, so $\mfg=\sum_{j}\mfg^{(j)}$ is a decomposition into simple ideals of $\mfg$.
Finally we remark that the Weyl group is generated by reflections on the roots, and therefore fixes the decomposition of $\mf{a}$.
Therefore the Weyl group is the product of the Weyl groups that act independently on each component of $\mf{a}$.
We finally obtain the following characterization.
\begin{prop} \label{soluzione}
    Suppose that $\mfg$ is semisimple, with the decomposition and notations above.
    Given any $X=\sum_j X_j\in\mf{a}$, let $I(X)=\{j: X_j\in\mf{a}^{(j)} \ \textit{is non-regular}\}$. Then $\mathcal{H}(X)=\bigoplus_{j\notin I(X)} \mf{a}^{(j)}$.
\end{prop}
\begin{proof}
    We start by saying that $\Delta(X)\cap \Delta^{(j)}\neq \varnothing$ if and only if $j\in I(X)$.
    For all $j\in I(X)$, we can look at the problem restricted to $\mf{a}^{(j)}$ to find by proposition \ref{radirred} that $\bigcap_{R(X)\cap\Delta^{(j)}} \ker_{\mf{a}^{(j)}}(\alpha) = \{0\}$.
    Therefore, in the global context this becomes $\bigcap_{R(X)\cap\Delta^{(j)}} \ker_{\mf{a}}(\alpha) = \{H=\sum_i H_i: H_j=0\}=\bigoplus_{i\neq j}\mf{a}^{(i)}$.
    On the contrary, for all $j\notin I(X)$ we have $R(X)\cap \Delta^{(j)}=\varnothing$, as the $\Delta^{(j)}$ are all orthogonal and $\Delta)(X)\cap\Delta^{(j)}=\varnothing$.
    Finally we find
    \begin{eqnarray*}
        \mathcal{H}(X) &=& \bigcap_{R(X)} \ker(\alpha)=\bigcap_{j\in I(X)} \bigcap_{R(X)\cap\Delta^{(j)}} \ker(\alpha) \\ &=& \bigcap_{j\in I(X)} \bigoplus_{i\neq j} \mf{a}^{(j)} = \bigoplus_{j\notin I(X)} \mf{a}^{(j)}.
    \end{eqnarray*}
\end{proof}
This concludes the complete characterization, which is rephrased in the following theorem, to be coherent with the notations of the discussion in Section \eqref{setting}. In particular, recall that 
\[
\mf{a}_{++}(\mc{C})= \{ X\in\mf{a}: \ \alpha(X)>0 \quad \forall\alpha\in (\mf{w}\Delta^+_{\mc{C}})\cap \Delta^+\} \subset \mf{a}
\]
is the maximal subset of $\mf{a}$ with uniform dynamical properties, in the sense of Section \ref{setting}, which also contains the Weyl chamber $\mf{a}_{++}$. 

\begin{thm}
    Suppose $\mfg$ is semisimple, and consider the decomposition above. Given $\mc{C}\subset \overline{\mf{a}_{++}}$, if $I(\mc{C})=\{ j: \ \mc{C}\subset \ker \alpha,\  \forall\alpha\in\Delta^{(j)}\}$, then
    \[
    \mf{a}_+(\mc{C})= \{ X\in\mf{a}: \ X \ \text{regular in }  \mfg_j, \ \forall j\notin I(\mc{C})\} = \bigoplus_{j\notin I(\mc{C})} \text{Reg}(\mfg_j) \oplus\bigoplus_{j\in I(\mc{C})} \mfg_j
    \]
\end{thm}
As a consequence, given $\mc{C}\subset \mfg$, we may take $G=\prod_{j\notin I(\mc{C})}G_j$ with Lie algebra $\tilde{\mf{g}}$ and define the projection $\Pi: \mfg\rightarrow \oplus_{j\notin I(\mc{C})} \mfg_i=\tilde{\mfg}$ according to the above decomposition, and pose $\tilde{\mc{C}}=\Pi \mc{C}$ . Then $F_{\mc{C}}\cong F_{\tilde{\mc{C}}}$ and the isomorphism conjugates the flow induced on $F_{\mc{C}}$ by $X\in\mf{a}$ to the flow induced on $F_{\tilde{\mc{C}}}$ by $\Pi X\in \Pi\mf{a}\subset \tilde{g}$.

\bibliographystyle{alpha}
\bibliography{refs}

\end{document}